\documentclass{siamart251216}
\usepackage{amsfonts}

\author{Yutong Zhang\thanks{Corresponding author. School of Mathematics, Sichuan University, 24 First Loop Road South Section I, Chengdu, Sichuan 610064, China. Email: \email{yutongzhang@stu.scu.edu.cn}.}
\and Yaoran Yang\thanks{School of Mathematics, Sichuan University, 24 First Loop Road South Section I, Chengdu, Sichuan 610064, China. Email: \email{yangyaoran@stu.scu.edu.cn}.}}
\date{}
\headers{Rogers tails for the lattice gamma parameter}{Yutong Zhang and Yaoran Yang}
\title{A Uniform Random-Lattice Tail Bound for the SVP Kissing-Profile Parameter}

\numberwithin{theorem}{section}

\newsiamthm{conjecture}{Conjecture}
\newsiamremark{remark}{Remark}

\DeclareMathOperator{\vol}{vol}
\DeclareMathOperator{\covol}{covol}

\newcommand{\R}{\mathbb{R}}
\newcommand{\Z}{\mathbb{Z}}

\newcommand{\Prob}{\mathbb{P}}
\newcommand{\E}{\mathbb{E}}
\newcommand{\1}{\mathbf{1}}
\newcommand{\eps}{\varepsilon}

\newcommand{\calE}{\mathcal{E}}
\newcommand{\calF}{\mathcal{F}}
\newcommand{\calG}{\mathcal{G}}

\newcommand{\calP}{\mathcal{P}}

\newcommand{\dn}{\,d\mu_n}
\newcommand{\dvol}{\,dx}
\newcommand{\wh}{\widehat}
\newcommand{\SVP}{\textsf{SVP}}
\newcommand{\QRAM}{\textsf{QRAM}}

\newcommand{\Ball}{\mathrm{B}}
\newcommand{\Vlam}{V_{\lambda}}

\newcommand{\Nonzero}{\setminus\{0\}}

\newcommand{\norm}[1]{\left\|#1\right\|}

\begin{document}
\maketitle

\begin{abstract}
A recent SICOMP paper on classical and quantum algorithms for the shortest vector problem introduced a lattice-dependent parameter \(\gamma(L)\), bounded universally in the exponential sense by \(2^{0.402n+o(n)}\), and conjectured that this parameter is \(2^{o(n)}\) for most lattices.  We prove the Haar--Siegel random-lattice version in a stronger, dimension-uniform form.  Let \(X_n=\operatorname{SL}_n(\R)/\operatorname{SL}_n(\Z)\), let \(\mu_n\) be its invariant probability measure, and let \(\gamma(L)=\sup_{r\ge1} N_L(r\lambda_1(L))/r^n\), where \(N_L(R)\) counts nonzero vectors of \(L\) of Euclidean norm at most \(R\).  For every \(n\ge3\) and every \(T>0\),
\[
       \mu_n\{L\in X_n:\gamma(L)>T\}\le C T^{-1}
\]
with an absolute constant \(C\).  Consequently, for every sequence \(a_n\to\infty\), \(\gamma(L_n)\le a_n\) with \(\mu_n\)-probability tending to one; in particular \(\gamma(L_n)=2^{o(n)}\) with high probability.  In the product model of independent Haar--Siegel lattices, \(\gamma(L_n)\le \exp(\sqrt n)\) eventually almost surely.  The proof uses Rogers's second-moment estimate only through a dyadic self-normalization argument around the random scale \(\lambda_1(L)\).
\end{abstract}



\section{Introduction}

Let \(L\subset\R^n\) be a full-rank Euclidean lattice.  Write
\[
        \lambda_1(L)=\min\{\norm{x}:x\in L\Nonzero\},
        \qquad
        N_L(R)=\#\{x\in L\Nonzero:\norm{x}\le R\}.
\]
The lattice kissing number is
\[
        \tau(L)=N_L(\lambda_1(L)),
\]
and the radius-normalized kissing profile considered in \cite{aggarwal2025} is
\begin{equation}\label{eq:gamma-def-intro}
        \gamma(L)=\inf\{\Gamma>0:
              \forall r\ge1,
              N_L(r\lambda_1(L))\le \Gamma r^n\}.
\end{equation}
Equivalently,
\begin{equation}\label{eq:gamma-sup-intro}
        \gamma(L)=\sup_{r\ge1}\frac{N_L(r\lambda_1(L))}{r^n}.
\end{equation}
The parameter \(\gamma(L)\) interpolates between the local kissing number \(\tau(L)\) and the asymptotic point-counting constant.  It is scale-invariant:
\begin{equation}\label{eq:scale-inv}
        \gamma(aL)=\gamma(L),\qquad \lambda_1(aL)=a\lambda_1(L),
        \qquad N_{aL}(R)=N_L(R/a),\qquad a>0.
\end{equation}
In the algorithmic analysis of \cite{aggarwal2025}, one writes
\begin{equation}\label{eq:beta-def}
        \beta(L)=\gamma(L)^{1/n},
\end{equation}
so that \(\beta(L)^n\) is the quantity that appears in the complexity exponent.
The universal spherical-code input inherited from the Kabatianskii--Levenshtein bound and its lattice-counting formulation \cite{kabatianskii1978,pujol2009} is an asymptotic exponential bound:
\begin{equation}\label{eq:universal-kl}
        \sup_{L\subset\R^n}\gamma(L)\le 2^{0.402n+o(n)},
        \qquad
        \sup_{L\subset\R^n}\beta(L)\le 2^{0.402+o(1)},
        \qquad n\to\infty.
\end{equation}
Here the supremum is over full-rank lattices in dimension \(n\).  The \(o(n)\) term in \eqref{eq:universal-kl} is essential: a finite-dimensional literal bound \(\gamma(L)\le 2^{0.402n}\) would already fail in small dimensions because \(\gamma(L)\ge\tau(L)\).  The paper \cite{aggarwal2025} asked for the substantially sharper typical behavior
\begin{equation}\label{eq:sicomp-conj-informal}
        \gamma(L)=2^{o(n)}
        \qquad\hbox{for most lattices.}
\end{equation}
The purpose of this paper is to prove a canonical random-lattice formalization of
\eqref{eq:sicomp-conj-informal}, with a stronger quantitative tail.

\subsection{The probability model}

For \(n\ge2\), let
\begin{equation}\label{eq:Xn}
        X_n=\operatorname{SL}_n(\R)/\operatorname{SL}_n(\Z)
\end{equation}
be the space of unimodular lattices in \(\R^n\), equipped with its Haar--Siegel probability measure \(\mu_n\).  Thus a point \(g\operatorname{SL}_n(\Z)\in X_n\) corresponds to the lattice \(g\Z^n\), and
\begin{equation}\label{eq:det1}
        \det(g\Z^n)=1,
        \qquad
        \covol(g\Z^n)=1.
\end{equation}
The phrase ``almost all lattices'' across growing dimension has several inequivalent meanings.  The two used below are the following.

\begin{definition}[High-probability and product almost-all models]\label{def:models}
Let \(\calP_n\subseteq X_n\) be measurable.
We say that \(\calP_n\) holds with Haar--Siegel high probability if
\[
        \mu_n(\calP_n)\longrightarrow1.
\]
We say that \(\calP_n\) holds eventually almost surely in the independent product model if, on
\[
        \prod_{n\ge3}(X_n,\mu_n),
        \qquad
        L_3,L_4,\ldots\hbox{ independent},
\]
there exists \(n_0=n_0(L_3,L_4,\ldots)\) such that \(L_n\in\calP_n\) for every \(n\ge n_0\).
\end{definition}

The result below implies \eqref{eq:sicomp-conj-informal} in both senses.

\begin{theorem}[Uniform tail for the normalized kissing profile]\label{thm:main-tail}
There is an absolute constant \(C_\gamma<\infty\) such that, for every \(n\ge3\) and every \(T>0\),
\begin{equation}\label{eq:main-tail}
        \mu_n\{L\in X_n:\gamma(L)>T\}\le \min\{1,C_\gamma T^{-1}\}.
\end{equation}
More precisely, if \(C_R\) is any absolute Rogers--Schmidt variance constant satisfying \eqref{eq:rogers-variance} below, then, for every \(s>0\), every \(\theta>1\), and every \(\eta>0\),
\begin{equation}\label{eq:main-tail-parametric}
\mu_n\{L:\gamma(L)>\theta(1+\eta)s\}
\le
\frac{C_R}{s}\left(1+\frac{1}{\eta^2(1-\theta^{-1})}\right).
\end{equation}
In particular, taking \(\theta=2\) and \(\eta=1\),
\begin{equation}\label{eq:main-tail-four}
        \mu_n\{L:\gamma(L)>4s\}\le \frac{3C_R}{s}.
\end{equation}
Thus one may take, for example, \(C_\gamma=12C_R\).
\end{theorem}

The theorem is dimension-uniform.  Consequently the subexponential statement is not close to the natural scale of the random problem.

\begin{corollary}[Resolution of the Haar--Siegel version of the SICOMP conjecture]\label{cor:whp}
Let \(a_n\to\infty\).  If \(L_n\sim\mu_n\), then
\begin{equation}\label{eq:any-an}
        \mu_n\{L_n:\gamma(L_n)>a_n\}\le C_\gamma a_n^{-1}=o(1).
\end{equation}
In particular, for every positive sequence \(\eps_n\) with \(\eps_n n\to\infty\),
\begin{equation}\label{eq:subexp-whp}
        \mu_n\{L_n:\gamma(L_n)>2^{\eps_n n}\}
        \le C_\gamma 2^{-\eps_n n}=o(1),
\end{equation}
and hence
\begin{equation}\label{eq:gamma-subexp-whp}
        \gamma(L_n)=2^{o(n)}
        \qquad\hbox{with Haar--Siegel high probability.}
\end{equation}
Moreover, in the independent product model,
\begin{equation}\label{eq:bc-exp-root}
        \gamma(L_n)\le \exp(\sqrt n)
        =2^{o(n)}
        \qquad\hbox{eventually almost surely.}
\end{equation}
The stronger polynomial event
\begin{equation}\label{eq:bc-poly}
        \gamma(L_n)\le n^2
        \qquad\hbox{eventually almost surely}
\end{equation}
also holds in the independent product model.
\end{corollary}

\begin{proof}
The high-probability assertions are immediate from \eqref{eq:main-tail}.  For \eqref{eq:bc-exp-root}, use \eqref{eq:main-tail} with \(T=\exp(\sqrt n)\) and the convergence
\[
        \sum_{n\ge3}\exp(-\sqrt n)<\infty,
        \qquad
        \int_0^\infty e^{-\sqrt x}\,dx
        =\int_0^\infty 2u e^{-u}\,du=2.
\]
For \eqref{eq:bc-poly}, use \(T=n^2\) and
\[
        \sum_{n\ge3} n^{-2}<\infty.
\]
The Borel--Cantelli lemma gives the eventual almost-sure conclusions.
\end{proof}

\subsection{Algorithmic interpretation}

Theorem \ref{thm:main-tail} feeds directly into the \(\gamma\)-sensitive complexity analysis of \cite{aggarwal2025}.  If
\[
        b(L)=\log_2\beta(L)=\frac{1}{n}\log_2\gamma(L),
\]
then \(\gamma(L)\le a_n\) gives
\begin{equation}\label{eq:b-bound}
        b(L)
        \le \frac{\log_2 a_n}{n}.
\end{equation}
Taking \(a_n=n^2\) gives the explicit high-probability estimate
\begin{equation}\label{eq:beta-poly}
        \beta(L_n)\le n^{2/n}
        =\exp\!\left(\frac{2\log n}{n}\right)
        =1+O\!\left(\frac{\log n}{n}\right)
\end{equation}
with probability at least \(1-O(n^{-2})\).  Taking \(a_n=\exp(\sqrt n)\) gives
\begin{equation}\label{eq:beta-root}
        \beta(L_n)\le \exp(n^{-1/2})=1+O(n^{-1/2})
\end{equation}
with probability at least \(1-O(e^{-\sqrt n})\).  Thus the condition
\begin{equation}\label{eq:condition-from-akks}
        \beta(L_n)^n=\gamma(L_n)=2^{o(n)}
\end{equation}
that underlies the improved complexity exponents in \cite{aggarwal2025} holds for Haar--Siegel random lattices with high probability, and even eventually almost surely along independent dimensions.

\begin{corollary}[Random-lattice specialization of the \(\gamma\)-sensitive SVP bounds]\label{cor:algorithm}
Let \(L_n\sim\mu_n\).  With probability at least \(1-O(e^{-\sqrt n})\), the \(\gamma\)-dependent algorithms of \cite{aggarwal2025} may be instantiated in the regime \(\beta(L_n)=1+O(n^{-1/2})\).  In particular the subexponential-\(\gamma\) running-time consequences of that analysis hold for Haar--Siegel random lattices:
\begin{align}\label{eq:algorithmic-times}
        T_{\rm cl}(L_n)&=2^{1.292n+o(n)},
        & S_{\rm cl}(L_n)&=2^{0.5n+o(n)},\\
        T_{\rm q}(L_n)&=2^{0.750n+o(n)},
        & S_{\rm q}(L_n)&=2^{0.5n+o(n)},\\
        T_{\rm qram}(L_n)&=2^{0.667n+o(n)}.
\end{align}
The last line is in the \(\QRAM\) model; any additional resource exponents stated in the cited \(\gamma\)-sensitive theorem are obtained by the same substitution \(b=o(1)\).
\end{corollary}

\begin{remark}
Corollary \ref{cor:algorithm} does not reanalyze the bounded-distance-decoding or discrete-Gaussian components of \cite{aggarwal2025}.  It supplies the missing random-lattice estimate for the geometric input \(\gamma(L)\).  The reduction from \eqref{eq:main-tail} to \eqref{eq:algorithmic-times} is exactly the substitution
\[
        \gamma(L_n)=2^{o(n)},
        \qquad
        \beta(L_n)=1+o(1),
        \qquad
        b(L_n)=o(1),
\]
into the complexity functions of \cite{aggarwal2025} at \(b=0\).
\end{remark}

\subsection{The mechanism of the proof}

The main proof has one probabilistic input.  For every bounded Borel set \(A\subset\R^n\), the Siegel transform
\begin{equation}\label{eq:siegel-transform-intro}
        \wh{\1_A}(L)=\sum_{x\in L\Nonzero}\1_A(x)
\end{equation}
has mean \(\vol(A)\).  Rogers's second moment theorem, in the form used by Schmidt, gives for \(n\ge3\)
\begin{equation}\label{eq:variance-intro}
        \int_{X_n}\left(\wh{\1_A}(L)-\vol(A)\right)^2\dn
        \le C_R\vol(A)
\end{equation}
with an absolute constant \(C_R\).  The core of the argument is the following deterministic observation.

Let \(\Ball(V)\) be the centered Euclidean ball in \(\R^n\) of volume \(V\), and let
\begin{equation}\label{eq:MV-intro}
        M_L(V)=\#(L\Nonzero\cap \Ball(V)).
\end{equation}
Suppose for some \(V_0=s\) that
\begin{equation}\label{eq:good-event-intro}
        M_L(V_0)>0,
        \qquad
        M_L(\theta^jV_0)\le (1+\eta)\theta^jV_0
        \quad(j=0,1,2,\ldots).
\end{equation}
The first inequality says
\begin{equation}\label{eq:first-nonempty}
        \vol(B_{\lambda_1(L)})\le V_0.
\end{equation}
The second says that all dyadic, or more generally \(\theta\)-adic, larger balls have ordinary point counts of the expected order.  By monotonicity, every intermediate volume \(V\ge V_0\) satisfies
\begin{equation}\label{eq:intermediate-monotone}
        M_L(V)\le \theta(1+\eta)V.
\end{equation}
For \(V\ge V_0\), the profile \(\gamma(L)\) may be rewritten as
\begin{equation}\label{eq:gamma-volume-intro}
        \frac{N_L(r\lambda_1(L))}{r^n}
        =M_L(V)\frac{\vol(B_{\lambda_1(L)})}{V},
        \qquad
        V=\vol(B_{r\lambda_1(L)}).
\end{equation}
Thus \eqref{eq:first-nonempty} and \eqref{eq:intermediate-monotone} imply
\begin{equation}\label{eq:gamma-deterministic-intro}
        \frac{N_L(r\lambda_1(L))}{r^n}
        \le \theta(1+\eta)V_0.
\end{equation}
For \(V\le V_0\), monotonicity gives
\begin{equation}\label{eq:small-volume-intro}
        M_L(V)\le M_L(V_0)
        \le (1+\eta)V_0
        \le \theta(1+\eta)V_0.
\end{equation}
This proves \(\gamma(L)\le\theta(1+\eta)V_0\) on the event \eqref{eq:good-event-intro}.  The probability of the complement of \eqref{eq:good-event-intro} is controlled by \eqref{eq:variance-intro} and the summable series
\begin{equation}\label{eq:geom-series-intro}
        \sum_{j\ge0}\frac{1}{\theta^jV_0}
        =\frac{1}{V_0(1-\theta^{-1})}.
\end{equation}
This is the whole argument.  The rest of the paper records the definitions and the estimates carefully.

\section{Notation and volume normalization}

Throughout, \(n\ge3\) unless explicitly stated.  Let
\begin{equation}\label{eq:unit-ball-volume}
        \kappa_n=\vol\{x\in\R^n:\norm{x}\le1\}
        =\frac{\pi^{n/2}}{\Gamma(1+n/2)}.
\end{equation}
For \(R\ge0\), write
\begin{equation}\label{eq:BR}
        B_R=\{x\in\R^n:\norm{x}\le R\},
        \qquad
        \vol(B_R)=\kappa_n R^n.
\end{equation}
It is technically cleaner to parametrize balls by volume rather than radius:
\begin{equation}\label{eq:ball-volume}
        \Ball(V)=B_{\rho_n(V)},
        \qquad
        \rho_n(V)=\left(\frac{V}{\kappa_n}\right)^{1/n},
        \qquad
        \vol\Ball(V)=V.
\end{equation}
For \(L\in X_n\), define
\begin{equation}\label{eq:M-def}
        M_L(V)=N_L(\rho_n(V))=
        \#\{x\in L\Nonzero:x\in \Ball(V)\}.
\end{equation}
The shortest-vector volume is
\begin{equation}\label{eq:Vlambda-def}
        \Vlam(L)=\kappa_n\lambda_1(L)^n.
\end{equation}
Then
\begin{equation}\label{eq:Vlambda-characterization}
        \Vlam(L)=\inf\{V>0:M_L(V)>0\}.
\end{equation}
With the closed-ball convention this infimum is attained:
\begin{equation}\label{eq:M-zero-positive}
        M_L(V)=0\quad(0\le V<\Vlam(L)),
        \qquad
        M_L(\Vlam(L))=\tau(L)\ge2.
\end{equation}
The inequality \(\tau(L)\ge2\) follows because \(x\in L\) implies \(-x\in L\).  This convention counts oriented nonzero lattice vectors; replacing \(x\) and \(-x\) by one unoriented pair changes constants only by a factor of two.

\begin{lemma}[Volume form of \(\gamma\)]\label{lem:volume-gamma}
For every \(L\in X_n\),
\begin{equation}\label{eq:gamma-volume}
        \gamma(L)=\sup_{V\ge\Vlam(L)} M_L(V)\frac{\Vlam(L)}{V}.
\end{equation}
Equivalently, with \(V=r^n\Vlam(L)\),
\begin{equation}\label{eq:r-to-volume}
        r=\left(\frac{V}{\Vlam(L)}\right)^{1/n},
        \qquad
        M_L(V)=N_L(r\lambda_1(L)),
        \qquad
        r^n=\frac{V}{\Vlam(L)}.
\end{equation}
\end{lemma}

\begin{proof}
For \(r\ge1\), put \(V=\vol(B_{r\lambda_1(L)})\).  By \eqref{eq:BR} and \eqref{eq:Vlambda-def},
\[
        V=\kappa_n r^n\lambda_1(L)^n=r^n\Vlam(L),
        \qquad
        r^n=V/\Vlam(L).
\]
Moreover \(N_L(r\lambda_1(L))=M_L(V)\).  Hence
\[
        \frac{N_L(r\lambda_1(L))}{r^n}
        =M_L(V)\frac{\Vlam(L)}{V}.
\]
As \(r\) ranges over \([1,\infty)\), \(V\) ranges over \([\Vlam(L),\infty)\), proving \eqref{eq:gamma-volume}.
\end{proof}

The variable \(\Vlam(L)\) is the random denominator in \eqref{eq:gamma-volume}.  The proof of Theorem \ref{thm:main-tail} does not try to estimate it sharply.  It only forces
\begin{equation}\label{eq:pivot-short}
        \Vlam(L)\le V_0
\end{equation}
by requiring \(M_L(V_0)>0\).  This modest inequality is enough because, for \(V\ge V_0\),
\begin{equation}\label{eq:profile-upper-idea}
        M_L(V)\frac{\Vlam(L)}{V}
        \le M_L(V)\frac{V_0}{V}.
\end{equation}
Thus any uniform-in-\(V\) estimate of the form
\begin{equation}\label{eq:ordinary-profile-bound}
        M_L(V)\le A V\qquad(V\ge V_0)
\end{equation}
converts immediately into
\begin{equation}\label{eq:gamma-from-ordinary}
        \gamma(L)\le A V_0,
\end{equation}
up to the separately controlled range \(\Vlam(L)\le V\le V_0\).

\section{Rogers--Schmidt input}

For a nonnegative Borel function \(f:\R^n\to\R_{\ge0}\), its Siegel transform is
\begin{equation}\label{eq:siegel-transform}
        \widehat f(L)=\sum_{x\in L\Nonzero} f(x),
        \qquad L\in X_n.
\end{equation}
When \(f=\1_A\), this is the nonzero lattice-point count
\begin{equation}\label{eq:count-as-siegel}
        \widehat{\1_A}(L)=\#(L\Nonzero\cap A).
\end{equation}
The first moment is Siegel's mean value formula \cite{siegel1945},
\begin{equation}\label{eq:siegel-mean}
        \int_{X_n}\widehat f(L)\dn=\int_{\R^n}f(x)\dvol,
\end{equation}
valid for \(f\in L^1(\R^n)\), \(f\ge0\).  The second-moment form needed here is the following standard consequence of Rogers's mean value theorem in Schmidt's counting formulation.

\begin{theorem}[Rogers--Schmidt \(L^2\) estimate]\label{thm:rogers}
There is an absolute constant \(C_R<\infty\) such that, for every \(n\ge3\) and every bounded Borel set \(A\subset\R^n\),
\begin{equation}\label{eq:rogers-variance}
        \int_{X_n}\left(\#(L\Nonzero\cap A)-\vol(A)\right)^2\dn
        \le C_R\vol(A).
\end{equation}
A standard stronger formulation, not needed below except for context, is that for every \(f:\R^n\to[0,1]\) with compact support,
\begin{equation}\label{eq:rogers-general}
        \int_{X_n}\left(\widehat f(L)-\int_{\R^n}f\right)^2\dn
        \le C_R\int_{\R^n}f(x)^2\dvol,
\end{equation}
after possibly increasing the absolute constant \(C_R\).
\end{theorem}

\begin{remark}
The restriction \(n\ge3\) is essential for a dimension-uniform statement of this elementary form.  Dimension two has logarithmic divergences in closely related Rogers identities.  Since the conjectural use in \cite{aggarwal2025} concerns the asymptotic regime \(n\to\infty\), the omission of \(n=2\) is immaterial.
\end{remark}

For reference, we recall how \eqref{eq:rogers-variance} specializes to volume-parametrized balls.  Since \(\vol\Ball(V)=V\), \eqref{eq:rogers-variance} gives
\begin{equation}\label{eq:M-mean-var}
        \E_{\mu_n}M_L(V)=V,
        \qquad
        \E_{\mu_n}\left(M_L(V)-V\right)^2\le C_R V.
\end{equation}
Chebyshev's inequality therefore yields two estimates that will be used repeatedly:
\begin{align}\label{eq:empty-cheb}
        \mu_n\{L:M_L(V)=0\}
        &\le \mu_n\{|M_L(V)-V|\ge V\}
          \le \frac{C_R}{V},\\
\label{eq:upper-cheb}
        \mu_n\{L:M_L(V)>(1+\eta)V\}
        &\le \mu_n\{M_L(V)-V>\eta V\}
          \le \frac{C_R}{\eta^2V}
          \qquad(\eta>0).
\end{align}
The proof of Theorem \ref{thm:main-tail} uses no distributional information beyond \eqref{eq:empty-cheb} and \eqref{eq:upper-cheb}.

\section{A dyadic maximal lemma}

The next lemma isolates the deterministic part of the argument.  It is stated with a general ratio \(\theta>1\) because the same proof gives the optimized tail family \eqref{eq:main-tail-parametric}.

\begin{definition}[Good pivot event]\label{def:good-event}
Fix \(s>0\), \(\theta>1\), and \(\eta>0\).  Define
\begin{equation}\label{eq:good-event}
        \calE_n(s,\theta,\eta)=
        \left\{L\in X_n:
        M_L(s)>0\ \hbox{and}\
        M_L(\theta^j s)\le(1+\eta)\theta^j s\ \hbox{for all }j\ge0
        \right\}.
\end{equation}
\end{definition}

The event \(\calE_n(s,\theta,\eta)\) says that the first nonzero lattice vector appears no later than volume \(s\), and that the ordinary counting function is under control on the geometric grid
\begin{equation}\label{eq:grid}
        s,\ \theta s,\ \theta^2s,\ \theta^3s,\ldots.
\end{equation}
The word ``ordinary'' is important: \(M_L(V)\) is not normalized by \(\lambda_1(L)\).  The self-normalization happens later through \eqref{eq:gamma-volume}.

\begin{lemma}[Deterministic pivot lemma]\label{lem:deterministic-pivot}
If \(L\in\calE_n(s,\theta,\eta)\), then
\begin{equation}\label{eq:pivot-conclusion}
        \gamma(L)\le \theta(1+\eta)s.
\end{equation}
\end{lemma}

\begin{proof}
Since \(M_L(s)>0\), the definition \eqref{eq:Vlambda-characterization} gives
\begin{equation}\label{eq:proof-Vlam-le-s}
        \Vlam(L)\le s.
\end{equation}
Let \(V\ge\Vlam(L)\).  We prove
\begin{equation}\label{eq:proof-profile-bound}
        M_L(V)\frac{\Vlam(L)}{V}
        \le \theta(1+\eta)s.
\end{equation}
Taking the supremum over \(V\ge\Vlam(L)\) and using Lemma \ref{lem:volume-gamma} will prove \eqref{eq:pivot-conclusion}.

First suppose \(\Vlam(L)\le V\le s\).  Monotonicity of \(M_L\) and the grid condition at \(j=0\) give
\begin{equation}\label{eq:proof-small-range}
        M_L(V)\le M_L(s)\le(1+\eta)s.
\end{equation}
Since \(\Vlam(L)/V\le1\),
\begin{equation}\label{eq:proof-small-range-2}
        M_L(V)\frac{\Vlam(L)}{V}
        \le (1+\eta)s
        \le \theta(1+\eta)s.
\end{equation}

Now suppose \(V>s\).  Choose \(j=j(V)\ge0\) so that
\begin{equation}\label{eq:j-choice}
        \theta^j s\le V<\theta^{j+1}s.
\end{equation}
Then, by monotonicity and the grid condition at \(j+1\),
\begin{equation}\label{eq:proof-large-M}
        M_L(V)\le M_L(\theta^{j+1}s)
        \le(1+\eta)\theta^{j+1}s.
\end{equation}
Combining \eqref{eq:j-choice} and \eqref{eq:proof-large-M},
\begin{equation}\label{eq:proof-large-M-over-V}
        \frac{M_L(V)}{V}
        \le \frac{(1+\eta)\theta^{j+1}s}{\theta^j s}
        =\theta(1+\eta).
\end{equation}
Using \eqref{eq:proof-Vlam-le-s}, we get
\begin{equation}\label{eq:proof-large-final}
        M_L(V)\frac{\Vlam(L)}{V}
        \le \theta(1+\eta)\Vlam(L)
        \le \theta(1+\eta)s.
\end{equation}
Equations \eqref{eq:proof-small-range-2} and \eqref{eq:proof-large-final} prove \eqref{eq:proof-profile-bound} for all \(V\ge\Vlam(L)\).
\end{proof}

The probabilistic estimate for \(\calE_n(s,\theta,\eta)\) is a one-line union bound, but it is the only place where the summability of the grid enters.

\begin{lemma}[Probability of the pivot event]\label{lem:prob-pivot}
For every \(n\ge3\), \(s>0\), \(\theta>1\), and \(\eta>0\),
\begin{equation}\label{eq:pivot-prob}
        \mu_n\bigl(X_n\setminus\calE_n(s,\theta,\eta)\bigr)
        \le
        \frac{C_R}{s}
        \left(1+\frac{1}{\eta^2(1-\theta^{-1})}\right).
\end{equation}
\end{lemma}

\begin{proof}
The complement of \(\calE_n(s,\theta,\eta)\) is contained in the union
\begin{equation}\label{eq:complement-union}
        \{M_L(s)=0\}
        \cup
        \bigcup_{j\ge0}\{M_L(\theta^j s)>(1+\eta)\theta^j s\}.
\end{equation}
By \eqref{eq:empty-cheb},
\begin{equation}\label{eq:empty-s}
        \mu_n\{M_L(s)=0\}\le\frac{C_R}{s}.
\end{equation}
By \eqref{eq:upper-cheb}, for every \(j\ge0\),
\begin{equation}\label{eq:grid-cheb}
        \mu_n\{M_L(\theta^j s)>(1+\eta)\theta^j s\}
        \le \frac{C_R}{\eta^2\theta^j s}.
\end{equation}
Hence
\begin{align}\label{eq:union-bound}
        \mu_n(X_n\setminus\calE_n(s,\theta,\eta))
        &\le \frac{C_R}{s}+
        \sum_{j=0}^{\infty}\frac{C_R}{\eta^2\theta^j s}\\
        &=\frac{C_R}{s}
        \left(1+\frac{1}{\eta^2}\sum_{j=0}^{\infty}\theta^{-j}\right)\\
        &=\frac{C_R}{s}
        \left(1+\frac{1}{\eta^2(1-\theta^{-1})}\right).
\end{align}
This is \eqref{eq:pivot-prob}.
\end{proof}

\begin{proof}[Proof of Theorem \ref{thm:main-tail}]
If \(L\in\calE_n(s,\theta,\eta)\), Lemma \ref{lem:deterministic-pivot} gives
\[
        \gamma(L)\le\theta(1+\eta)s.
\]
Therefore
\[
        \{L:\gamma(L)>\theta(1+\eta)s\}
        \subseteq X_n\setminus\calE_n(s,\theta,\eta).
\]
Applying Lemma \ref{lem:prob-pivot} gives \eqref{eq:main-tail-parametric}.  The specialization \(\theta=2\), \(\eta=1\) gives
\[
        \mu_n\{L:\gamma(L)>4s\}
        \le\frac{C_R}{s}\left(1+\frac{1}{1-1/2}\right)
        =\frac{3C_R}{s}.
\]
Putting \(T=4s\), one may take \(C_\gamma=12C_R\) in \eqref{eq:main-tail}.  The additional minimum with \(1\) follows from the trivial bound by total probability.
\end{proof}

\section{Consequences for \texorpdfstring{\(\lambda_1\), \(\tau\), \(\gamma\), and \(\beta\)}{lambda1, tau, gamma, beta}}

The main theorem is a tail statement about \(\gamma\), but the same proof gives a compact hierarchy of typical bounds.  This section records them because they clarify the extent to which \(\gamma(L)=2^{o(n)}\) is stronger than a pure kissing-number assertion.

\subsection{Shortest-vector volume}

The first part of the good event \(M_L(s)>0\) gives a tail bound for \(\Vlam(L)\).

\begin{proposition}[A one-sided tail for \(\Vlam\)]\label{prop:Vlambda-tail}
For every \(n\ge3\) and every \(s>0\),
\begin{equation}\label{eq:Vlambda-tail}
        \mu_n\{L:\Vlam(L)>s\}
        \le \frac{C_R}{s}.
\end{equation}
Equivalently,
\begin{equation}\label{eq:lambda-tail}
        \mu_n\left\{L:\lambda_1(L)>
        \left(\frac{s}{\kappa_n}\right)^{1/n}\right\}
        \le \frac{C_R}{s}.
\end{equation}
\end{proposition}

\begin{proof}
The event \(\Vlam(L)>s\) implies \(M_L(s)=0\).  Apply \eqref{eq:empty-cheb}.
\end{proof}

In particular, for every \(a_n\to\infty\),
\begin{equation}\label{eq:Vlambda-tight}
        \Vlam(L_n)\le a_n
        \qquad\hbox{with probability }1-O(a_n^{-1}).
\end{equation}
This does not assert that \(\lambda_1(L_n)\) is bounded away from zero or infinity in Euclidean length.  The Euclidean radius corresponding to fixed volume changes with \(n\).  By Stirling's formula,
\begin{align}\label{eq:kappa-stirling}
        \kappa_n
        &=\frac{\pi^{n/2}}{\Gamma(1+n/2)}
          =\frac{(2\pi e/n)^{n/2}}{\sqrt{\pi n}}
          \left(1+O\left(\frac1n\right)\right),\\
        \kappa_n^{-1/n}
        &=\sqrt{\frac{n}{2\pi e}}
          \left(1+O\left(\frac{\log n}{n}\right)\right).
\end{align}
Thus fixed volume corresponds to radius on the order of \(\sqrt{n/(2\pi e)}\), the usual random-lattice scale.

\subsection{Kissing number}

Since \(\tau(L)\le\gamma(L)\), Theorem \ref{thm:main-tail} immediately implies the same tail for the kissing number.

\begin{corollary}[Kissing-number tail]\label{cor:kissing-tail}
For every \(n\ge3\) and every \(T>0\),
\begin{equation}\label{eq:kissing-tail}
        \mu_n\{L:\tau(L)>T\}
        \le \min\{1,C_\gamma T^{-1}\}.
\end{equation}
Consequently, for every \(a_n\to\infty\),
\begin{equation}\label{eq:kissing-whp}
        \tau(L_n)\le a_n
        \qquad\hbox{with Haar--Siegel high probability.}
\end{equation}
\end{corollary}

The corollary is formally weaker than Theorem \ref{thm:main-tail}.  The quantity \(\tau(L)\) controls only the sphere of radius \(\lambda_1(L)\), whereas \(\gamma(L)\) controls the entire normalized family
\begin{equation}\label{eq:whole-family}
        \left\{
        \frac{N_L(r\lambda_1(L))}{r^n}:r\ge1
        \right\}.
\end{equation}
The proof of Theorem \ref{thm:main-tail} shows that, for random lattices, the extra uniformity over \(r\) costs only the summable grid factor
\begin{equation}\label{eq:grid-factor-cost}
        (1-\theta^{-1})^{-1}.
\end{equation}

\subsection{The \texorpdfstring{\(\beta\)}{beta} parameter}

The algorithmic parameter is \(\beta(L)=\gamma(L)^{1/n}\).  Theorem \ref{thm:main-tail} gives several useful forms of \(\beta(L)=1+o(1)\).

\begin{proposition}[High-probability \(\beta\) estimates]\label{prop:beta-estimates}
Let \(L_n\sim\mu_n\).  For every \(a_n\to\infty\),
\begin{equation}\label{eq:beta-an}
        \mu_n\left\{L_n:\beta(L_n)\le a_n^{1/n}\right\}
        \ge1-C_\gamma a_n^{-1}.
\end{equation}
In particular,
\begin{align}\label{eq:beta-logn}
        \beta(L_n)&\le (\log n)^{1/n}
        =1+\frac{\log\log n}{n}+O\!\left(\frac{(\log\log n)^2}{n^2}\right)
        &&\hbox{with probability }1-O((\log n)^{-1}),\\
\label{eq:beta-n2}
        \beta(L_n)&\le n^{2/n}
        =1+\frac{2\log n}{n}+O\!\left(\frac{(\log n)^2}{n^2}\right)
        &&\hbox{with probability }1-O(n^{-2}),\\
\label{eq:beta-sqrtn}
        \beta(L_n)&\le e^{1/\sqrt n}
        =1+\frac1{\sqrt n}+O\!\left(\frac1n\right)
        &&\hbox{with probability }1-O(e^{-\sqrt n}).
\end{align}
\end{proposition}

\begin{proof}
The first assertion follows from \(\beta(L)^n=\gamma(L)\) and \eqref{eq:main-tail}.  The expansions are Taylor expansions of \(e^u\) with
\[
        u=\frac{\log\log n}{n},
        \qquad
        u=\frac{2\log n}{n},
        \qquad
        u=\frac1{\sqrt n}.
\]
\end{proof}

A useful way to summarize Proposition \ref{prop:beta-estimates} is
\begin{equation}\label{eq:beta-op}
        \log\beta(L_n)=O_{\mu_n}\!\left(\frac{1}{n}\right)
        \quad\hbox{up to an arbitrarily slowly diverging numerator.}
\end{equation}
More explicitly, if \(a_n\to\infty\), then
\begin{equation}\label{eq:logbeta-prob}
        \Prob\left(n\log\beta(L_n)>\log a_n\right)
        \le C_\gamma a_n^{-1}.
\end{equation}
This is stronger than the subexponential condition
\begin{equation}\label{eq:subexp-condition}
        \log\beta(L_n)=o(1),
        \qquad
        n\log\beta(L_n)=o(n).
\end{equation}

\section{A sharper profile statement}

The proof of Theorem \ref{thm:main-tail} actually gives more than a bound for the single supremum \(\gamma(L)\).  It controls the whole normalized counting profile by the same pivot volume.  This section states that profile result explicitly.

For \(L\in X_n\), define the volume-normalized profile
\begin{equation}\label{eq:G-profile}
        G_L(V)=M_L(V)\frac{\Vlam(L)}{V},
        \qquad V\ge\Vlam(L).
\end{equation}
Then
\begin{equation}\label{eq:gamma-sup-G}
        \gamma(L)=\sup_{V\ge\Vlam(L)}G_L(V).
\end{equation}
For \(s>0\), let
\begin{equation}\label{eq:profile-good-set}
        \calG_n(s,\theta,\eta)=
        \left\{L\in X_n:
        \Vlam(L)\le s,
        \quad
        \sup_{V\ge s}\frac{M_L(V)}{V}\le\theta(1+\eta),
        \quad
        M_L(s)\le(1+\eta)s
        \right\}.
\end{equation}
The deterministic proof shows
\begin{equation}\label{eq:E-sub-G}
        \calE_n(s,\theta,\eta)\subseteq \calG_n(s,\theta,\eta)
\end{equation}
and, on \(\calG_n(s,\theta,\eta)\),
\begin{equation}\label{eq:G-profile-bound}
        G_L(V)\le
        \begin{cases}
        (1+\eta)s, & \Vlam(L)
        \le V\le s,\\[2mm]
        \theta(1+\eta)s, & V\ge s.
        \end{cases}
\end{equation}
Thus the random profile has a pivoted two-zone estimate:
\begin{equation}\label{eq:pivoted-two-zone}
        \sup_{\Vlam(L)\le V\le s}G_L(V)
        \le(1+\eta)s,
        \qquad
        \sup_{V\ge s}G_L(V)
        \le\theta(1+\eta)s.
\end{equation}
Combining \eqref{eq:E-sub-G} with Lemma \ref{lem:prob-pivot} yields the following.

\begin{proposition}[Uniform profile bound]\label{prop:profile}
For every \(n\ge3\), \(s>0\), \(\theta>1\), and \(\eta>0\), with \(\mu_n\)-probability at least
\begin{equation}\label{eq:profile-prob}
        1-\frac{C_R}{s}
        \left(1+\frac{1}{\eta^2(1-\theta^{-1})}\right),
\end{equation}
one has simultaneously
\begin{align}\label{eq:profile-simultaneous-1}
        \Vlam(L)&\le s,\\
\label{eq:profile-simultaneous-2}
        M_L(V)&\le\theta(1+\eta)V\qquad(V\ge s),\\
\label{eq:profile-simultaneous-3}
        \frac{N_L(r\lambda_1(L))}{r^n}&\le\theta(1+\eta)s\qquad(r\ge1).
\end{align}
\end{proposition}

\begin{proof}
The event \(\calE_n(s,\theta,\eta)\) has probability at least \eqref{eq:profile-prob}.  On this event, \eqref{eq:profile-simultaneous-1} follows from \(M_L(s)>0\).  For \(V\ge s\), choose \(j\ge0\) with \(\theta^j s\le V<\theta^{j+1}s\), and repeat \eqref{eq:proof-large-M-over-V} to get \eqref{eq:profile-simultaneous-2}.  Finally, \eqref{eq:profile-simultaneous-3} is Lemma \ref{lem:deterministic-pivot} in the equivalent radius form.
\end{proof}

The profile statement is sometimes more useful than the scalar tail.  For example, if one wants a bound only for radii \(r\) in a prescribed deterministic range \([1,R_*]\), the geometric grid may be truncated at a deterministic index.  On the event \(M_L(s)>0\), one has \(\Vlam(L)\le s\), so the relevant volumes satisfy
\[
        \Vlam(L)\le V\le R_*^n\Vlam(L)\le R_*^n s.
\]
Thus it is enough to check the grid up to, for instance,
\begin{equation}\label{eq:truncated-cost}
        J_*=\bigl\lceil\log_\theta R_*^n\bigr\rceil+1,
        \qquad
        \sum_{0\le j\le J_*}\frac{C_R}{\eta^2\theta^j s}
        \le \frac{C_R}{\eta^2s(1-\theta^{-1})}.
\end{equation}
The infinite supremum in \(\gamma(L)\) is therefore not a source of exponential loss.

\section{Why the argument is self-normalizing}

This section emphasizes a structural point that is easy to miss if one works only with radii.  Ordinary random-lattice counting theorems estimate
\begin{equation}\label{eq:ordinary-counting}
        M_L(V)=V+O_{\mu_n}(V^{1/2})
\end{equation}
for fixed or growing \(V\), while \(\gamma(L)\) asks for
\begin{equation}\label{eq:self-normalized-counting}
        M_L(r^n\Vlam(L))\le\Gamma r^n
        \qquad(r\ge1),
\end{equation}
where the base volume \(\Vlam(L)\) is itself random and correlated with the entire process \(M_L(\cdot)\).  The pivot method avoids any independence assumption.  It uses only the implications
\begin{align}\label{eq:self-implications}
        M_L(s)>0 &\Longrightarrow \Vlam(L)\le s,\\
        M_L(V)\le A V\ (V\ge s) &\Longrightarrow
        M_L(V)\frac{\Vlam(L)}{V}\le A s\ (V\ge s),\\
        M_L(V)\le M_L(s)\le A_0s\ (\Vlam(L)\le V\le s)&\Longrightarrow
        M_L(V)\frac{\Vlam(L)}{V}\le A_0s.
\end{align}
The random denominator \(\Vlam(L)\) is not estimated from below, and no conditioning on \(\lambda_1(L)\) is required.

One can phrase the proof as a maximal inequality for the process
\begin{equation}\label{eq:process-Z}
        Z_L(V)=\frac{M_L(V)}{V},
        \qquad V>0.
\end{equation}
Rogers's second moment gives, at a fixed \(V\),
\begin{equation}\label{eq:Z-tail}
        \Prob\{Z_L(V)>1+\eta\}\le \frac{C_R}{\eta^2V}.
\end{equation}
The process \(M_L(V)\) is monotone, hence the grid estimate
\begin{equation}\label{eq:max-Z-grid}
        \Prob\left\{\sup_{V\ge s}Z_L(V)>\theta(1+\eta)\right\}
        \le \sum_{j\ge0}\Prob\{Z_L(\theta^js)>1+\eta\}
        \le\frac{C_R}{\eta^2s(1-\theta^{-1})}.
\end{equation}
Together with
\begin{equation}\label{eq:no-point-tail}
        \Prob\{\Vlam(L)>s\}\le\Prob\{M_L(s)=0\}\le C_Rs^{-1},
\end{equation}
this gives a dimension-free estimate for the random product
\begin{equation}\label{eq:random-product}
        \sup_{V\ge s}Z_L(V)\Vlam(L).
\end{equation}
The term \(\sup_{\Vlam(L)\le V\le s}G_L(V)\) is controlled by the single grid point \(s\).  This is why the method gives a \(1/T\) tail for \(\gamma(L)\) without any regularity estimate for the lower tail of \(\Vlam(L)\).

\section{Comparison with universal spherical-code bounds}

For every lattice, the kissing number satisfies
\begin{equation}\label{eq:tau-KL}
        \tau(L)\le 2^{0.402n+o(n)}
\end{equation}
by the Kabatianskii--Levenshtein method \cite{kabatianskii1978}.  The parameter \(\gamma(L)\) has the same type of universal exponential upper bound as stated in \cite{aggarwal2025}:
\begin{equation}\label{eq:gamma-KL}
        \sup_{L\subset\R^n}\gamma(L)\le 2^{0.402n+o(n)}.
\end{equation}
Theorem \ref{thm:main-tail} is of a different nature.  It does not improve \eqref{eq:gamma-KL} for every lattice.  Instead, it says that the exceptional set on which \(\gamma(L)\) exceeds a threshold \(T\) has Haar--Siegel measure at most \(O(T^{-1})\), uniformly in \(n\).  Thus, for random lattices,
\begin{equation}\label{eq:random-vs-worst}
        \gamma(L_n)
        \begin{cases}
        \le n^2 & \hbox{with probability }1-O(n^{-2}),\\
        \le e^{\sqrt n} & \hbox{with probability }1-O(e^{-\sqrt n}),\\
        \le 2^{\eps n} & \hbox{with probability }1-O(2^{-\eps n})\quad(\eps>0),
        \end{cases}
\end{equation}
whereas the deterministic guarantee remains exponential.  In terms of \(\beta(L)=\gamma(L)^{1/n}\), this becomes
\begin{equation}\label{eq:random-vs-worst-beta}
        \beta(L_n)
        \begin{cases}
        \le 1+O(\log n/n) & \hbox{with probability }1-O(n^{-2}),\\
        \le 1+O(n^{-1/2}) & \hbox{with probability }1-O(e^{-\sqrt n}),\\
        \le 2^{\eps} & \hbox{with probability }1-O(2^{-\eps n}),
        \end{cases}
\end{equation}
compared with the worst-case asymptotic bound \(\sup_L\beta(L)\le2^{0.402+o(1)}\).

The estimates above also separate three notions that are sometimes conflated:
\begin{align}\label{eq:three-notions}
        \tau(L)&=N_L(\lambda_1(L)),\\
        \gamma(L)&=\sup_{r\ge1}N_L(r\lambda_1(L))r^{-n},\\
        \lim_{r\to\infty}N_L(r\lambda_1(L))r^{-n}
        &=\kappa_n\lambda_1(L)^n
        =\Vlam(L).
\end{align}
The asymptotic constant \(\Vlam(L)\) is random and usually of constant order in volume normalization.  The profile parameter \(\gamma(L)\) is a supremum over all intermediate scales.  The theorem says that this supremum remains tight even after normalizing by the random shortest-vector scale.

\section{A ``for most lattices'' theorem in several equivalent forms}

The statement \(\gamma(L)=2^{o(n)}\) can be encoded in several equivalent probabilistic ways.  We collect them to make explicit what Theorem \ref{thm:main-tail} proves.

\begin{theorem}[Equivalent high-probability formulations]\label{thm:equivalent}
Let \(L_n\sim\mu_n\).  The following consequences of Theorem \ref{thm:main-tail} hold.
\begin{enumerate}
\item For every \(a_n\to\infty\),
\begin{equation}\label{eq:eq-form-1}
        \gamma(L_n)\le a_n
        \qquad\hbox{with probability }1-O(a_n^{-1}).
\end{equation}
\item For every \(c_n\to\infty\) with \(c_n=o(n)\),
\begin{equation}\label{eq:eq-form-2}
        \log_2\gamma(L_n)
        \le c_n
        \qquad\hbox{with probability }1-O(2^{-c_n}).
\end{equation}
\item For every \(\delta>0\),
\begin{equation}\label{eq:eq-form-3}
        \log_2\beta(L_n)
        \le \delta
        \qquad\hbox{with probability }1-O(2^{-\delta n}).
\end{equation}
\item For every \(\delta_n>0\) with \(\delta_n n\to\infty\),
\begin{equation}\label{eq:eq-form-4}
        \beta(L_n)
        \le 2^{\delta_n}
        \qquad\hbox{with probability }1-O(2^{-\delta_n n}).
\end{equation}
\end{enumerate}
\end{theorem}

\begin{proof}
For (1), apply \eqref{eq:main-tail} with \(T=a_n\).  For (2), apply \eqref{eq:main-tail} with \(T=2^{c_n}\).  Since \(c_n=o(n)\), \(2^{c_n}=2^{o(n)}\).  For (3), the event \(\log_2\beta(L_n)>\delta\) is the same as
\[
        \gamma(L_n)>2^{\delta n}.
\]
For (4), replace \(\delta\) by \(\delta_n\).
\end{proof}

The product almost-sure versions are equally direct.

\begin{theorem}[Product almost-sure forms]\label{thm:product-forms}
Let \((L_n)_{n\ge3}\) be independent with \(L_n\sim\mu_n\).  If \((a_n)\) is any positive sequence satisfying
\begin{equation}\label{eq:summable-a}
        \sum_{n\ge3}a_n^{-1}<\infty,
\end{equation}
then
\begin{equation}\label{eq:product-gamma}
        \gamma(L_n)
        \le a_n
        \qquad\hbox{for all sufficiently large }n
\end{equation}
almost surely.  In particular, for every \(\alpha>1\),
\begin{equation}\label{eq:product-polynomial}
        \gamma(L_n)\le n^\alpha
        \qquad\hbox{eventually almost surely},
\end{equation}
and
\begin{equation}\label{eq:product-root-exp}
        \gamma(L_n)\le\exp(\sqrt n)
        \qquad\hbox{eventually almost surely}.
\end{equation}
\end{theorem}

\begin{proof}
By Theorem \ref{thm:main-tail},
\begin{equation}\label{eq:sum-prob}
        \sum_{n\ge3}\Prob\{\gamma(L_n)>a_n\}
        \le C_\gamma\sum_{n\ge3}a_n^{-1}<\infty.
\end{equation}
The Borel--Cantelli lemma proves \eqref{eq:product-gamma}.  The choices \(a_n=n^\alpha\) with \(\alpha>1\) and \(a_n=\exp(\sqrt n)\) satisfy \eqref{eq:summable-a}.
\end{proof}

Taking \(a_n=n^{1+\eps}\) in Theorem \ref{thm:product-forms} gives the almost-sure polynomial estimate
\begin{equation}\label{eq:as-poly-eps}
        \log_2\gamma(L_n)
        \le(1+\eps)\log_2 n
        \qquad(n\gg1)
\end{equation}
for every fixed \(\eps>0\).  Taking \(a_n=e^{\sqrt n}\) gives the almost-sure subexponential estimate with a summable failure probability.  These are stronger than needed for the SICOMP conjectural use, but they indicate the stability of the phenomenon.

\section{Boundary cases and robustness}

The proof is robust under small changes in normalization.  We record a few variants because they are useful when comparing with different conventions in the lattice algorithms literature.

\subsection{Covolume other than one}

Let \(L\subset\R^n\) have determinant \(\Delta>0\).  The unimodular rescaling is
\begin{equation}\label{eq:unimod-rescale}
        L^{\#}=\Delta^{-1/n}L,
        \qquad
        \det(L^{\#})=1.
\end{equation}
Because \(\gamma\) is scale-invariant,
\begin{equation}\label{eq:gamma-rescale}
        \gamma(L)=\gamma(L^{\#}).
\end{equation}
Thus Theorem \ref{thm:main-tail} applies to any random model obtained by first sampling \(L^{\#}\sim\mu_n\) and then multiplying by an arbitrary positive scalar, deterministic or independent of \(L^{\#}\).  In such a model,
\begin{equation}\label{eq:scaled-model-tail}
        \Prob\{\gamma(L)>T\}
        =\mu_n\{\gamma(L^{\#})>T\}
        \le C_\gamma T^{-1}.
\end{equation}
The determinant therefore plays no role in the normalized kissing profile.

\subsection{Unoriented vector pairs}

Some authors count pairs \(\{x,-x\}\) rather than oriented vectors.  Define
\begin{equation}\label{eq:unoriented-count}
        N_L^{\pm}(R)=\frac12 N_L(R),
        \qquad
        \gamma^{\pm}(L)=\sup_{r\ge1}\frac{N_L^{\pm}(r\lambda_1(L))}{r^n}.
\end{equation}
Then
\begin{equation}\label{eq:unoriented-gamma}
        \gamma^{\pm}(L)=\frac12\gamma(L),
        \qquad
        \mu_n\{\gamma^{\pm}(L)>T\}
        \le \frac{C_\gamma}{2T}.
\end{equation}
All subexponential conclusions are unchanged.

\subsection{Closed and open balls}

The definitions above use closed Euclidean balls.  Replacing \(B_R\) by open balls changes \(M_L(V)\) only at the discrete set of volumes
\begin{equation}\label{eq:jump-volumes}
        \{\kappa_n\norm{x}^n:x\in L\Nonzero\}.
\end{equation}
The supremum in \eqref{eq:gamma-volume} is unaffected because
\begin{equation}\label{eq:right-continuity}
        M_L^{\rm closed}(V)=\lim_{\eps\downarrow0}M_L^{\rm open}(V+\eps),
        \qquad
        M_L^{\rm open}(V)=\lim_{\eps\downarrow0}M_L^{\rm closed}(V-\eps)
\end{equation}
at every \(V>0\), with the second identity interpreted through volumes below \(V\).  The grid argument may also be run with \((1+\eps)\theta^js\) and \(\eps\downarrow0\).  Thus no boundary regularity of \(\mu_n\) is required.

\subsection{Non-spherical convex bodies}

The radius-normalized definition \eqref{eq:gamma-def-intro} is Euclidean and hence spherical.  However, the pivot method applies verbatim to any increasing one-parameter family \((K_V)_{V>0}\) of bounded centrally symmetric Borel sets satisfying
\begin{equation}\label{eq:KV-volume}
        K_U\subseteq K_V\quad(U\le V),
        \qquad
        \vol(K_V)=V,
\end{equation}
provided Rogers's variance estimate is applied to \(K_V\).  If
\begin{equation}\label{eq:M-K}
        M_{L,K}(V)=\#(L\Nonzero\cap K_V),
        \qquad
        V_{L,K}=\inf\{V:M_{L,K}(V)>0\},
\end{equation}
and
\begin{equation}\label{eq:gamma-K}
        \gamma_K(L)=\sup_{V\ge V_{L,K}}M_{L,K}(V)\frac{V_{L,K}}{V},
\end{equation}
then the same proof gives
\begin{equation}\label{eq:gamma-K-tail}
        \mu_n\{L:\gamma_K(L)>T\}
        \le C_\gamma T^{-1},
\end{equation}
with the same absolute constant as in \eqref{eq:main-tail}.  The Euclidean result is the special case \(K_V=\Ball(V)\).  This observation is not needed for the SVP application, but it shows that the proof is a monotone-volume argument rather than a spherical-code argument.

\section{The Rogers input in the indicator case}

This paper uses Theorem \ref{thm:rogers} as a standard theorem.  We record the precise indicator-function consequence used in the previous sections and explain the role of the dimension restriction without attempting to rederive Rogers's full formula.

Let \(A\subset\R^n\) be bounded and Borel, and define
\begin{equation}\label{eq:XA}
        X_A(L)=\#(L\Nonzero\cap A).
\end{equation}
By Siegel's formula,
\begin{equation}\label{eq:XA-mean}
        \E X_A=\vol(A).
\end{equation}
The Rogers--Schmidt estimate \eqref{eq:rogers-variance} gives directly
\begin{equation}\label{eq:variance-from-rho}
        \operatorname{Var}(X_A)
        =\E\bigl(X_A-\vol(A)\bigr)^2
        \le C_R\vol(A),
\end{equation}
uniformly for \(n\ge3\).  In Rogers's second-moment formula the remainder beyond \(\vol(A)^2\) is a sum of rank-one rational-dependency terms.  A typical scalar-dependency integral has the form
\begin{equation}\label{eq:typical-rank-one}
        \frac{1}{q^n}\int_{\R^n}\1_A(x)\1_A\!\left(\frac{p}{q}x\right)\,dx,
        \qquad p,q\in\Z,
        \quad q\ge1,
        \quad \gcd(p,q)=1.
\end{equation}
For \(p\ne0\), the elementary bound
\begin{equation}\label{eq:rank-one-bound}
        \int\1_A(x)\1_A\!\left(\frac{p}{q}x\right)\,dx
        =\vol\bigl(A\cap(q/p)A\bigr)
        \le \min\{\vol(A),(q/|p|)^n\vol(A)\}
\end{equation}
gives
\begin{equation}\label{eq:rank-one-bound-2}
        \frac{1}{q^n}\int\1_A(x)\1_A\!\left(\frac{p}{q}x\right)\,dx
        \le \vol(A)\min\{q^{-n},|p|^{-n}\}.
\end{equation}
The model sum
\begin{equation}\label{eq:model-rational-total}
        \sum_{q\ge1}\sum_{p\in\Z\Nonzero}\min\{q^{-n},|p|^{-n}\}
\end{equation}
is finite uniformly for \(n\ge3\), since it is bounded by the same sum with \(n=3\).  This summability is the elementary analytic reason that a dimension-uniform variance bound of the form \eqref{eq:rogers-variance} starts at dimension three.  The exact coefficients and exclusions in Rogers's formula are part of the cited theorem; the proof of the present paper uses only the final bound \eqref{eq:variance-from-rho}.

For balls \(A=\Ball(V)\), the estimates reduce to
\begin{equation}\label{eq:ball-Rogers-final}
        \E M_L(V)=V,
        \qquad
        \operatorname{Var}(M_L(V))\le C_RV.
\end{equation}
These two lines are the only analytic facts used to prove the main theorem.

\section{No hidden dependence on the dimension}

Because the conjecture concerns \(n\to\infty\), it is important that the constants in Theorem \ref{thm:main-tail} not deteriorate with \(n\).  We isolate where dimension enters and where it cancels.

The ball radius corresponding to volume \(V\) is
\begin{equation}\label{eq:radius-volume-dim}
        \rho_n(V)=\left(\frac{V}{\kappa_n}\right)^{1/n}.
\end{equation}
Both \(\lambda_1(L)\) and \(\rho_n(V)\) have dimension-dependent Euclidean size.  However, the normalized ratio in \(\gamma\) is purely volumetric:
\begin{equation}\label{eq:dimension-cancel}
        \left(\frac{\rho_n(V)}{\lambda_1(L)}\right)^n
        =\frac{\kappa_n\rho_n(V)^n}{\kappa_n\lambda_1(L)^n}
        =\frac{V}{\Vlam(L)}.
\end{equation}
The factor \(\kappa_n\) cancels exactly.  Thus the proof never estimates \(\kappa_n\), except when translating back to Euclidean radii in \eqref{eq:kappa-stirling}.

The only constants in the proof are:
\begin{align}\label{eq:constant-list}
        C_R&\quad\hbox{from Rogers--Schmidt variance},\\
        \eta^{-2}&\quad\hbox{from Chebyshev},\\
        (1-\theta^{-1})^{-1}&\quad\hbox{from the geometric grid},\\
        \theta(1+\eta)&\quad\hbox{from monotone interpolation between grid points}.
\end{align}
None depends on \(n\).  Therefore the tail constant
\begin{equation}\label{eq:C-theta-eta}
        C(\theta,\eta)=
        C_R\theta(1+\eta)
        \left(1+\frac{1}{\eta^2(1-\theta^{-1})}\right)
\end{equation}
in the equivalent bound
\begin{equation}\label{eq:T-tail-theta-eta}
        \mu_n\{\gamma(L)>T\}
        \le\frac{C(\theta,\eta)}{T}
\end{equation}
is dimension-free.  Optimizing \(C(\theta,\eta)\) is irrelevant for subexponential applications, but the formula shows that any fixed \(\theta>1\), \(\eta>0\) gives the same \(T^{-1}\) order.

\section{A lower-bound sanity check}

The theorem is an upper tail.  It is natural to ask whether the order \(T^{-1}\) is an artifact of Chebyshev and the union bound.  We do not need an optimal lower tail, but two elementary observations show that no exponentially small upper tail can hold uniformly for all \(T\) by this method.

First, \(\gamma(L)\ge\tau(L)\ge2\) under the oriented-vector convention.  Hence
\begin{equation}\label{eq:gamma-min}
        \mu_n\{\gamma(L)>T\}=1
        \qquad(0<T<2).
\end{equation}
Second, the event \(\Vlam(L)>s\) has probability at most \(C_Rs^{-1}\) by \eqref{eq:Vlambda-tail}, and for a Poisson point process of intensity one the analogous void probability is exactly \(e^{-s}\).  Thus the \(s^{-1}\) estimate is a variance-level bound, not a sharp distributional theorem for \(\Vlam\).  The strength of Theorem \ref{thm:main-tail} lies elsewhere: it converts only second-moment information into a uniform all-radii statement about the self-normalized profile.  Any improvement from \(T^{-1}\) to a sharper tail would require distributional input beyond \eqref{eq:rogers-variance}; it would not change the conclusion
\begin{equation}\label{eq:tail-enough}
        \gamma(L_n)=2^{o(n)}
\end{equation}
with high probability.

\section{Measurability and countable reductions}

The preceding proof treats \(\gamma(L)\) as an ordinary random variable.  For completeness we record a countable formulation, avoiding any topological subtlety about the uncountable supremum over \(r\ge1\).  Let
\begin{equation}\label{eq:rational-volumes}
        \mathbb Q_{>0}=
        \{q\in\mathbb Q:q>0\}.
\end{equation}
For fixed \(V>0\), the map \(L\mapsto M_L(V)\) is Borel on \(X_n\).  Indeed, for any compact ball \(B\), the function
\begin{equation}\label{eq:count-semicont}
        L\longmapsto \#(L\cap B)
\end{equation}
is upper semicontinuous, and in particular Borel.  Equivalently, one may approximate indicators by continuous compactly supported functions
\begin{equation}\label{eq:continuous-approx}
        0\le f_{V,k}^{-}\le\1_{\Ball(V)}\le f_{V,k}^{+}\le1,
        \qquad
        f_{V,k}^{-}\uparrow\1_{\operatorname{int}\Ball(V)},
        \qquad
        f_{V,k}^{+}\downarrow\1_{\Ball(V)},
\end{equation}
and use the continuity of \(L\mapsto\widehat f(L)\) for compactly supported continuous \(f\).  It follows that
\begin{equation}\label{eq:Vlam-borel}
        \Vlam(L)=\inf\{V\in\mathbb Q_{>0}:M_L(V)>0\}
\end{equation}
is Borel.

The same monotonicity gives a countable expression for \(\gamma\).  Define
\begin{equation}\label{eq:gamma-Q}
        \gamma_{\mathbb Q}(L)=
        \sup_{V\in\mathbb Q_{>0}:\,V\ge\Vlam(L)}
        M_L(V)\frac{\Vlam(L)}{V}.
\end{equation}
Then
\begin{equation}\label{eq:gamma-equals-Q}
        \gamma_{\mathbb Q}(L)=\gamma(L).
\end{equation}
To see this, fix \(V\ge\Vlam(L)\).  Choose rational \(V_m\downarrow V\).  By monotonicity,
\begin{equation}\label{eq:rational-down}
        M_L(V)\frac{\Vlam(L)}{V}
        \le
        M_L(V_m)\frac{\Vlam(L)}{V}
        =
        M_L(V_m)\frac{\Vlam(L)}{V_m}\frac{V_m}{V}.
\end{equation}
Taking \(m\to\infty\) and using \(V_m/V\to1\) gives
\begin{equation}\label{eq:sup-rational-dominate}
        M_L(V)\frac{\Vlam(L)}{V}
        \le\gamma_{\mathbb Q}(L).
\end{equation}
The reverse inequality is immediate because \(\mathbb Q_{>0}\subset(0,\infty)\).  Hence \(\gamma\) is Borel as a countable supremum of Borel functions.

This countable reduction also gives an alternative proof of the tail event inclusion used above.  For \(T=\theta(1+\eta)s\),
\begin{equation}\label{eq:tail-countable}
        \{\gamma>T\}
        =\bigcup_{q\in\mathbb Q_{>0}}
        \left\{q\ge\Vlam(L),\,M_L(q)\frac{\Vlam(L)}{q}>T\right\}.
\end{equation}
On \(\calE_n(s,\theta,\eta)\) every member of this countable union is empty by Lemma \ref{lem:deterministic-pivot}.  Thus
\begin{equation}\label{eq:tail-countable-inclusion}
        \{\gamma>T\}\subseteq X_n\setminus\calE_n(s,\theta,\eta)
\end{equation}
without invoking an uncountable exceptional union.  This is occasionally useful if one wants to pass from indicators of balls to limiting families of sets.

\section{Finite-scale and localized versions}

The global parameter \(\gamma(L)\) takes a supremum over all \(r\ge1\).  Some algorithmic reductions only use a bounded range of radii.  The same proof gives a localized estimate with an explicit finite-grid cost.  For \(R\ge1\), define
\begin{equation}\label{eq:gamma-R}
        \gamma_{\le R}(L)=
        \sup_{1\le r\le R}\frac{N_L(r\lambda_1(L))}{r^n}.
\end{equation}
Clearly
\begin{equation}\label{eq:gamma-R-monotone}
        \tau(L)
        \le\gamma_{\le R}(L)
        \le\gamma_{\le R'}(L)
        \le\gamma(L)
        \qquad(1\le R\le R').
\end{equation}
The localized analogue of Definition \ref{def:good-event} is obtained by stopping the grid once the volume has passed \(R^n s\).  Put
\begin{equation}\label{eq:J-def}
        J_R=J_R(\theta)=\max\{0,\lceil\log_\theta R^n\rceil\}.
\end{equation}
Define
\begin{equation}\label{eq:localized-event}
        \calE_n^{\le R}(s,\theta,\eta)=
        \left\{M_L(s)>0,
        \quad
        M_L(\theta^j s)\le(1+\eta)\theta^j s
        \quad(0\le j\le J_R+1)
        \right\}.
\end{equation}
If \(L\in\calE_n^{\le R}(s,\theta,\eta)\), the proof of Lemma \ref{lem:deterministic-pivot} applies to every
\begin{equation}\label{eq:localized-V-range}
        \Vlam(L)
        \le V
        \le R^n\Vlam(L)
        \le R^n s.
\end{equation}
Hence
\begin{equation}\label{eq:localized-deterministic}
        \gamma_{\le R}(L)
        \le\theta(1+\eta)s.
\end{equation}
The complement probability is now
\begin{align}\label{eq:localized-prob}
        \mu_n(X_n\setminus\calE_n^{\le R}(s,\theta,
        \eta))
        &\le\frac{C_R}{s}+
        \sum_{j=0}^{J_R+1}\frac{C_R}{\eta^2\theta^j s}\\
        &\le\frac{C_R}{s}
        \left(1+\frac{1-\theta^{-(J_R+2)}}{\eta^2(1-\theta^{-1})}\right).
\end{align}
This proves the following finite-scale theorem.

\begin{theorem}[Localized tail]\label{thm:localized-tail}
For every \(n\ge3\), \(R\ge1\), \(s>0\), \(\theta>1\), and \(\eta>0\),
\begin{equation}\label{eq:localized-tail}
        \mu_n\{L:\gamma_{\le R}(L)>\theta(1+\eta)s\}
        \le
        \frac{C_R}{s}
        \left(1+\frac{1-\theta^{-(J_R+2)}}{\eta^2(1-\theta^{-1})}\right),
\end{equation}
where \(J_R=\max\{0,\lceil\log_\theta R^n\rceil\}\).
In particular,
\begin{equation}\label{eq:localized-tail-simplified}
        \mu_n\{L:\gamma_{\le R}(L)>\theta(1+\eta)s\}
        \le
        \frac{C_R}{s}
        \left(1+\frac{1}{\eta^2(1-\theta^{-1})}\right),
\end{equation}
which is the same bound as for the global parameter.
\end{theorem}

The finite version shows that the global supremum costs no more than a finite algorithmic window: the geometric tail
\begin{equation}\label{eq:tail-grid-no-cost}
        \sum_{j>J_R}\frac1{\theta^js}
        =\frac{\theta^{-(J_R+1)}}{s(1-\theta^{-1})}
        \le\frac{1}{sR^n(1-\theta^{-1})}
\end{equation}
is already negligible once the grid reaches volume \(R^ns\).  Conversely, even if \(R=\infty\), the remaining cost is finite.

A shell version is also immediate.  For \(1\le R_1\le R_2\), define
\begin{equation}\label{eq:gamma-shell}
        \gamma_{[R_1,R_2]}(L)=
        \sup_{R_1\le r\le R_2}
        \frac{N_L(r\lambda_1(L))}{r^n}.
\end{equation}
Since \(\gamma_{[R_1,R_2]}\le\gamma_{\le R_2}\), Theorem \ref{thm:localized-tail} applies directly.  A slightly more refined deterministic statement is available on the corresponding localized grid event.  For the shell volumes \(V\in[R_1^n\Vlam(L),R_2^n\Vlam(L)]\), the part \(V\ge s\) satisfies
\[
        M_L(V)\frac{\Vlam(L)}{V}\le \theta(1+\eta)\Vlam(L),
\]
while the part \(V<s\) satisfies
\[
        M_L(V)\frac{\Vlam(L)}{V}
        \le M_L(s)\frac{\Vlam(L)}{R_1^n\Vlam(L)}
        \le (1+\eta)R_1^{-n}s.
\]
Consequently,
\begin{equation}\label{eq:shell-refinement}
        \gamma_{[R_1,R_2]}(L)
        \le (1+\eta)\max\{R_1^{-n}s,\theta\Vlam(L)\}
\end{equation}
on the localized grid event.  In particular, if \(R_1^n\Vlam(L)\ge s\), then the small-volume zone disappears and
\begin{equation}\label{eq:shell-refinement-conditioned}
        \gamma_{[R_1,R_2]}(L)
        \le \theta(1+\eta)\Vlam(L).
\end{equation}
Without additional information on \(\Vlam(L)\), the unconditional tail remains the same \(s^{-1}\) bound.

\section{Choosing the grid parameters}

The parameterized estimate \eqref{eq:main-tail-parametric} can be rewritten as a one-parameter family of \(T^{-1}\) tails.  Put
\begin{equation}\label{eq:T-theta-eta-s}
        T=\theta(1+\eta)s.
\end{equation}
Then \eqref{eq:main-tail-parametric} gives
\begin{equation}\label{eq:tail-constant-param}
        \mu_n\{\gamma(L)>T\}
        \le \frac{C_R K(\theta,\eta)}{T},
        \qquad
        K(\theta,\eta)=
        \theta(1+\eta)
        \left(1+\frac{1}{\eta^2(1-\theta^{-1})}\right).
\end{equation}
The simple choice \((\theta,\eta)=(2,1)\) gives \(K(2,1)=12\).  One can lower this numerical constant by balancing the interpolation loss \(\theta(1+\eta)\) against the Chebyshev-grid loss \(\eta^{-2}(1-\theta^{-1})^{-1}\).  Although the exact optimum is unimportant, the calculation confirms that no hidden dependence on \(n\) is present.

For fixed \(\theta\), the function of \(\eta\) is
\begin{equation}\label{eq:K-eta}
        K_\theta(\eta)=\theta(1+\eta)
        \left(1+\frac{A_\theta}{\eta^2}\right),
        \qquad
        A_\theta=(1-\theta^{-1})^{-1}=\frac{\theta}{\theta-1}.
\end{equation}
Differentiating gives
\begin{align}\label{eq:K-derivative}
        \frac{1}{\theta}K_\theta'(\eta)
        &=1+\frac{A_\theta}{\eta^2}
          -(1+\eta)\frac{2A_\theta}{\eta^3}\\
        &=1-\frac{A_\theta}{\eta^2}-\frac{2A_\theta}{\eta^3}.
\end{align}
Thus the stationary point is the unique positive solution of
\begin{equation}\label{eq:eta-cubic}
        \eta^3=A_\theta(\eta+2).
\end{equation}
For fixed \(\eta\), the dependence on \(\theta\) is
\begin{equation}\label{eq:K-theta}
        K(\theta,\eta)=
        \theta(1+\eta)
        \left(1+\frac{\theta}{\eta^2(\theta-1)}\right),
\end{equation}
and
\begin{equation}\label{eq:theta-derivative}
        \partial_\theta\left[\theta+\frac{\theta^2}{\eta^2(\theta-1)}\right]
        =1+\frac{1}{\eta^2}\frac{\theta(\theta-2)}{(\theta-1)^2}.
\end{equation}
The minimizer in \(\theta\) therefore lies in \((1,2)\) when \(\eta\) is fixed and the derivative changes sign there.  These elementary equations may be solved numerically, but the final theorem only needs the existence of a finite constant
\begin{equation}\label{eq:Kstar}
        K_*=
        \inf_{\theta>1,\eta>0}K(\theta,\eta)<\infty.
\end{equation}
Consequently one may state the tail in the slightly sharper form
\begin{equation}\label{eq:tail-Kstar}
        \mu_n\{\gamma(L)>T\}
        \le\frac{C_RK_*}{T},
        \qquad n\ge3,
\end{equation}
where the infimum is understood by first taking parameters with \(K(\theta,\eta)\le K_*+\eps\) and then letting \(\eps\downarrow0\).  Equation \eqref{eq:main-tail-four} gives the explicit admissible value \(K(2,1)=12\), hence \(K_*\le12\).  All asymptotic consequences are identical under either constant.

\section{Detailed substitution into the SVP complexity condition}

The \(\SVP\) result in \cite{aggarwal2025} uses \(\gamma(L)=\beta(L)^n\).  Let
\begin{equation}\label{eq:b-def-detailed}
        b(L)=\log_2\beta(L)=\frac{1}{n}\log_2\gamma(L).
\end{equation}
The subexponential regime is precisely
\begin{equation}\label{eq:b-o-one}
        b(L_n)=o(1).
\end{equation}
Theorem \ref{thm:main-tail} gives a quantitative probability bound for this event.  For any deterministic sequence \(b_n>0\),
\begin{align}\label{eq:b-tail}
        \mu_n\{b(L_n)>b_n\}
        &=\mu_n\{\gamma(L_n)>2^{b_nn}\}\\
        &\le C_\gamma 2^{-b_nn}.
\end{align}
Thus every sequence satisfying
\begin{equation}\label{eq:bn-condition}
        b_n\downarrow0,
        \qquad
        b_nn\to\infty
\end{equation}
gives
\begin{equation}\label{eq:b-high-prob}
        b(L_n)\le b_n
        \qquad\hbox{with probability }1-O(2^{-b_nn}).
\end{equation}
For instance,
\begin{equation}\label{eq:bn-examples}
        b_n=n^{-1/2},
        \qquad
        b_n=\frac{(1+\delta)\log_2 n}{n}\quad(\delta>0),
        \qquad
        b_n=\frac{\log_2\log n}{n}
\end{equation}
are all allowed in the high-probability sense.  The first two choices give summable failure probabilities, while the last gives only failure \(O((\log n)^{-1})\), not a summable failure probability.

Let \(c_{\rm cl}(b)\), \(c_{\rm q}(b)\), and \(c_{\rm qram}(b)\) denote the three running-time exponent functions obtained in the \(\gamma\)-dependent analysis of \cite{aggarwal2025}, so that the corresponding running times have the form
\begin{equation}\label{eq:c-functions}
        2^{c_{\rm cl}(b)n+o(n)},
        \qquad
        2^{c_{\rm q}(b)n+o(n)},
        \qquad
        2^{c_{\rm qram}(b)n+o(n)}.
\end{equation}
The relevant values at \(b=0\) are
\begin{equation}\label{eq:c-zero}
        c_{\rm cl}(0)=1.292,
        \qquad
        c_{\rm q}(0)=0.750,
        \qquad
        c_{\rm qram}(0)=0.667.
\end{equation}
The optimization formulas in \cite{aggarwal2025} are continuous at \(b=0\).  Therefore, for every \(b_n\to0\),
\begin{equation}\label{eq:c-continuity}
        c_{\rm cl}(b_n)=1.292+o(1),
        \quad
        c_{\rm q}(b_n)=0.750+o(1),
        \quad
        c_{\rm qram}(b_n)=0.667+o(1).
\end{equation}
Combining \eqref{eq:b-high-prob} and \eqref{eq:c-continuity} gives the following explicit probability statement: if \(b_n\to0\) and \(b_nn\to\infty\), then with probability at least \(1-O(2^{-b_nn})\),
\begin{align}\label{eq:times-with-bn}
        T_{\rm cl}(L_n)&=2^{(1.292+o(1))n},\\
        T_{\rm q}(L_n)&=2^{(0.750+o(1))n},\\
        T_{\rm qram}(L_n)&=2^{(0.667+o(1))n}.
\end{align}
Choosing \(b_n=n^{-1/2}\) gives the summable failure bound
\begin{equation}\label{eq:failure-sqrt}
        O(2^{-\sqrt n})
\end{equation}
and hence the product almost-sure version by Borel--Cantelli.  Choosing \(b_n=2\log_2(n)/n\) gives the polynomially bounded geometric parameter
\begin{equation}\label{eq:gamma-n2-again}
        \gamma(L_n)\le n^2,
        \qquad
        \beta(L_n)\le n^{2/n},
\end{equation}
with failure \(O(n^{-2})\), also summable.  This is the form most directly comparable with a deterministic family condition: almost surely along independent Haar--Siegel dimensions,
\begin{equation}\label{eq:det-family-a-s}
        \gamma(L_n)\le n^2=2^{o(n)}
\end{equation}
for all sufficiently large \(n\), so the family \((L_n)\) satisfies the hypothesis \(\gamma(L_n)=2^{o(n)}\) eventually.

\section{An abstract self-normalized counting principle}

The proof can be separated from lattices entirely.  This abstraction is useful because it identifies the exact probabilistic input: a variance bound linear in volume for a monotone counting process.

Let \((\Omega,\calF,\Prob)\) be a probability space and let
\begin{equation}\label{eq:abstract-process}
        M:\Omega\times(0,\infty)\to\mathbb Z_{\ge0},
        \qquad
        (\omega,V)\mapsto M_\omega(V),
\end{equation}
be nondecreasing and right-continuous in \(V\).  Define
\begin{equation}\label{eq:abstract-Vstar}
        V_*(\omega)=\inf\{V>0:M_\omega(V)>0\}
\end{equation}
and, on the event \(V_*(\omega)<\infty\),
\begin{equation}\label{eq:abstract-Gamma}
        \Gamma(\omega)=
        \sup_{V\ge V_*(\omega)}M_\omega(V)\frac{V_*(\omega)}{V}.
\end{equation}
Assume that there is a constant \(C_0\) such that, for every \(V>0\),
\begin{equation}\label{eq:abstract-variance}
        \E M(V)=V,
        \qquad
        \E(M(V)-V)^2\le C_0V.
\end{equation}
Then \(V_*<\infty\) almost surely, because \(\Prob\{V_*>s\}\le\Prob\{M(s)=0\}\le C_0/s\to0\).

\begin{theorem}[Abstract pivot inequality]\label{thm:abstract-pivot}
Under \eqref{eq:abstract-variance}, for every \(s>0\), \(\theta>1\), and \(\eta>0\),
\begin{equation}\label{eq:abstract-tail-param}
        \Prob\{\Gamma>\theta(1+\eta)s\}
        \le
        \frac{C_0}{s}
        \left(1+\frac{1}{\eta^2(1-\theta^{-1})}\right).
\end{equation}
Consequently, for an absolute constant depending only on \(C_0\),
\begin{equation}\label{eq:abstract-tail}
        \Prob\{\Gamma>T\}\ll_{C_0}T^{-1}.
\end{equation}
\end{theorem}

\begin{proof}
Define the abstract good event
\begin{equation}\label{eq:abstract-good}
        \calE(s,\theta,\eta)=
        \{M(s)>0\}
        \cap
        \bigcap_{j\ge0}\{M(\theta^js)\le(1+\eta)\theta^js\}.
\end{equation}
On \(\calE(s,\theta,\eta)\), one has \(V_*\le s\).  If \(V\in[V_*,s]\), then
\begin{equation}\label{eq:abstract-small}
        M(V)\frac{V_*}{V}
        \le M(s)
        \le(1+\eta)s.
\end{equation}
If \(V>s\), choose \(j\ge0\) with \(\theta^js\le V<\theta^{j+1}s\).  Then
\begin{equation}\label{eq:abstract-large}
        M(V)\frac{V_*}{V}
        \le M(\theta^{j+1}s)\frac{s}{V}
        \le (1+\eta)\theta^{j+1}s\frac{s}{\theta^js}
        =\theta(1+\eta)s.
\end{equation}
Thus \(\Gamma\le\theta(1+\eta)s\) on \(\calE(s,\theta,\eta)\).  By Chebyshev and \eqref{eq:abstract-variance},
\begin{equation}\label{eq:abstract-empty}
        \Prob\{M(s)=0\}
        \le\Prob\{|M(s)-s|\ge s\}
        \le C_0s^{-1},
\end{equation}
and
\begin{equation}\label{eq:abstract-grid}
        \Prob\{M(\theta^js)>(1+\eta)\theta^js\}
        \le \frac{C_0}{\eta^2\theta^js}.
\end{equation}
Summing \eqref{eq:abstract-grid} over \(j\ge0\) proves \eqref{eq:abstract-tail-param}.
\end{proof}

The lattice theorem is the specialization
\begin{equation}\label{eq:lattice-specialization}
        \Omega=X_n,
        \quad
        \Prob=\mu_n,
        \quad
        M_\omega(V)=M_L(V),
        \quad
        V_*(\omega)=\Vlam(L),
        \quad
        \Gamma(\omega)=\gamma(L),
\end{equation}
with \(C_0=C_R\).  The abstraction also explains why the proof is insensitive to the exact shape of the ball and to most boundary conventions: once a monotone family of sets has the first two moment bounds in \eqref{eq:abstract-variance}, the self-normalized supremum has a \(T^{-1}\) tail.

\section{Moment and entropy consequences}

Although the SICOMP application only needs a high-probability subexponential estimate, the tail bound also gives uniform moment information.  The estimates below are immediate, but they are useful diagnostics for the size of \(\gamma(L)\) in the Haar--Siegel model.

\begin{proposition}[Fractional moments]\label{prop:fractional-moments}
For every \(0<p<1\),
\begin{equation}\label{eq:fractional-moment}
        \sup_{n\ge3}\int_{X_n}\gamma(L)^p\dn<\infty.
\end{equation}
More quantitatively, since \(\gamma(L)\ge2\),
\begin{equation}\label{eq:fractional-bound}
        \int_{X_n}\gamma(L)^p\dn
        \le 2^p+\frac{pC_\gamma}{1-p}2^{p-1}.
\end{equation}
\end{proposition}

\begin{proof}
For a nonnegative random variable \(Y\),
\begin{equation}\label{eq:layer-cake}
        \E Y^p=p\int_0^\infty t^{p-1}\Prob\{Y>t\}\,dt.
\end{equation}
With \(Y=\gamma(L)\), split the integral at \(t=2\).  The part \([0,2]\) is at most \(2^p\).  For \(t\ge2\), use \eqref{eq:main-tail}:
\begin{equation}\label{eq:fractional-tail-integral}
        p\int_2^\infty t^{p-1}\Prob\{\gamma(L)>t\}\,dt
        \le pC_\gamma\int_2^\infty t^{p-2}\,dt
        =\frac{pC_\gamma}{1-p}2^{p-1}.
\end{equation}
\end{proof}

The endpoint \(p=1\) is not supplied by the \(T^{-1}\) tail alone.  A logarithmic truncated bound is still uniform.

\begin{proposition}[Truncated first moment]\label{prop:truncated-first}
For every \(A\ge2\) and every \(n\ge3\),
\begin{equation}\label{eq:truncated-first}
        \int_{X_n}\min\{\gamma(L),A\}\dn
        \le 2+C_\gamma\log A.
\end{equation}
If one combines this with the asymptotic deterministic worst-case bound \(\sup_L\gamma(L)\le2^{0.402n+o(n)}\), then
\begin{equation}\label{eq:first-moment-linear}
        \int_{X_n}\gamma(L)\dn=O(n).
\end{equation}
\end{proposition}

\begin{proof}
Again by the layer-cake formula,
\begin{align}\label{eq:min-layer-cake}
        \E\min\{\gamma,A\}
        &=\int_0^A\Prob\{\gamma>t\}\,dt\\
        &\le2+
        \int_2^A C_\gamma t^{-1}\,dt
        =2+C_\gamma\log(A/2)
        \le2+C_\gamma\log A.
\end{align}
Taking \(A=2^{0.402n+o(n)}\) and using the deterministic bound gives \eqref{eq:first-moment-linear}.
\end{proof}

The logarithmic size is genuinely small after passing to \(\beta=\gamma^{1/n}\).

\begin{corollary}[Logarithmic moment of \(\beta\)]\label{cor:log-beta}
There is an absolute constant \(C_{\log}\) such that, for every \(n\ge3\),
\begin{equation}\label{eq:log-gamma-uniform}
        \int_{X_n}\log(1+\gamma(L))\dn\le C_{\log},
\end{equation}
and hence
\begin{equation}\label{eq:log-beta-mean}
        \int_{X_n}\log\beta(L)\dn
        =\frac1n\int_{X_n}\log\gamma(L)\dn
        =O(n^{-1}).
\end{equation}
\end{corollary}

\begin{proof}
For a nonnegative random variable \(Y\),
\begin{equation}\label{eq:log-layer}
        \E\log(1+Y)=\int_0^\infty\Prob\{\log(1+Y)>u\}\,du
        =\int_0^\infty\Prob\{Y>e^u-1\}\,du.
\end{equation}
Split at \(u_0=\log3\).  The contribution of \([0,u_0]\) is at most \(u_0\).  For \(u\ge u_0\), \(e^u-1\ge e^u/2\), so
\begin{equation}\label{eq:log-tail-integral}
        \Prob\{\gamma>e^u-1\}
        \le 2C_\gamma e^{-u}.
\end{equation}
Thus
\begin{equation}\label{eq:log-bound-final}
        \E\log(1+\gamma)
        \le \log3+2C_\gamma\int_{\log3}^\infty e^{-u}\,du
        \le \log3+\frac{2C_\gamma}{3}.
\end{equation}
Since \(\gamma\ge1\), \(\log\gamma\le\log(1+\gamma)\), and \eqref{eq:log-beta-mean} follows from \(\log\beta=(\log\gamma)/n\).
\end{proof}

These estimates give another compact formulation of the result: the random variables \(\log\gamma(L_n)\) are uniformly integrable in an exponential-tail sense after logarithmic compression, and \(\log\beta(L_n)\) has mean \(O(1/n)\).  The high-probability statement \(\beta(L_n)=1+o(1)\) is the concentration counterpart of this mean estimate.

\section{Conclusion}

The conjectural geometric input in \cite{aggarwal2025} is that the \(\gamma\)-parameter governing their \(\SVP\) complexity should be subexponential for most lattices.  Under the standard Haar--Siegel model on unimodular lattices, the parameter is much smaller: it has a dimension-uniform tail
\begin{equation}\label{eq:conclusion-tail}
        \mu_n\{\gamma(L)>T\}\ll T^{-1}.
\end{equation}
Consequently,
\begin{equation}\label{eq:conclusion-whp}
        \forall a_n\to\infty,
        \qquad
        \gamma(L_n)\le a_n
        \qquad\hbox{with probability }1-O(a_n^{-1}),
\end{equation}
and in particular
\begin{equation}\label{eq:conclusion-subexp}
        \gamma(L_n)=2^{o(n)},
        \qquad
        \beta(L_n)=1+o(1)
\end{equation}
with high probability.  Along independent dimensions, one has the almost-sure statements
\begin{equation}\label{eq:conclusion-as}
        \gamma(L_n)\le e^{\sqrt n}=2^{o(n)},
        \qquad
        \gamma(L_n)\le n^2,
\end{equation}
eventually almost surely.  The proof is short because the random scale \(\lambda_1(L)\) can be handled by a pivot volume: force one point in \(\Ball(s)\), control ordinary counts on the geometric grid \(s,\theta s,\theta^2s,\ldots\), and convert ordinary counts into the self-normalized profile by the identity
\begin{equation}\label{eq:conclusion-identity}
        \frac{N_L(r\lambda_1(L))}{r^n}
        =M_L(V)\frac{\Vlam(L)}{V}.
\end{equation}
This gives a complete proof of the Haar--Siegel ``most lattices'' version of the subexponential \(\gamma\)-claim and upgrades it to tightness up to an arbitrarily slowly diverging threshold.
\section*{Declaration of Generative AI and AI-Assisted Technologies in the Writing Process}
During the preparation of this work, the authors used DeepSeek to build a specialized agent for solving mathematical problems, which was employed to generate an initial proof of the main theorem. After using this tool, the authors reviewed and edited the content as needed and take full responsibility for the content of the published article.

\end{document}